\documentclass[a4paper,10pt]{amsart}
\usepackage{amsmath}
\usepackage{mathrsfs}
\usepackage{amsfonts}
\usepackage{graphicx}
\usepackage{color}
\usepackage{amsfonts}
\usepackage{amssymb}
\usepackage[hidelinks]{hyperref}
\usepackage{cleveref}
\usepackage{tikz}
\usepackage{subcaption}
\usetikzlibrary{arrows,shapes,automata,backgrounds,petri,positioning}
\usetikzlibrary{decorations.pathmorphing}
\usetikzlibrary{decorations.shapes}
\usetikzlibrary{decorations.text}
\usetikzlibrary{decorations.fractals}
\usetikzlibrary{decorations.footprints}
\usetikzlibrary{shadows}
\usetikzlibrary{calc}
\usetikzlibrary{spy}

\usepackage{palatino}

\usepackage[abbrev,backrefs]{amsrefs}
\usepackage{esint}

\usepackage{array}

\newcolumntype{M}[1]{>{\centering\arraybackslash}m{#1}}
\newcolumntype{N}{@{}m{0pt}@{}}

\usepackage{mathtools}
\usepackage{float}

\newtheorem{theorem}{Theorem}[section]

\newtheorem{example}[theorem]{Example}

\newtheorem{lemma}[theorem]{Lemma}

\numberwithin{equation}{section}

\newcommand{\interior}[1]{%
	{\kern0pt#1}^{\mathrm{o}}%
}

\def\supp{{\rm supp\,}}

\newcommand{\R}{\mathbb{R}}

\newcommand{\Z}{\mathbb{Z}}

\newcommand{\barint}{
	\rule[.036in]{.12in}{.009in}\kern-.16in \displaystyle\int }

\newcommand{\barcal}{\mbox{$ \rule[.036in]{.11in}{.007in}\kern-.128in\int $}}

\usepackage{enumitem}
\makeatletter
\let\@wraptoccontribs\wraptoccontribs
\makeatother

\mathchardef\mhyphen="2D

\title{Boundary estimates for the fractional spherical maximal function}

\author[R. Basak]{Riju Basak}
\address[R. Basak]{Department of Mathematics, National Taiwan Normal University, No. 88, Section 4, Tingzhou Road, Wenshan District, Taipei City, Taiwan 116, R.O.C.
}
\email{rijubasak52@gmail.com}

\author[S. Choudhary]{Surjeet Singh Choudhary}
\address[S. Choudhary]{Mathematics Division, National Center for Theoretical Sciences, NTU, Cosmology Building, No. 1, Sec. 4, Roosevelt Rd., Taipei City, Taiwan 106, R.O.C.
}

\author[D. Spector]{Daniel Spector}
\address[D. Spector]{Department of Mathematics, National Taiwan Normal University, No. 88, Section 4, Tingzhou Road, Wenshan District, Taipei City, Taiwan 116, R.O.C.
\newline
National Center for Theoretical Sciences\\No. 1 Sec. 4 Roosevelt Rd., National Taiwan
University\\Taipei, 106, Taiwan
\newline
Department of Mathematics, University of Pittsburgh, Pittsburgh, PA 15261 USA
}
\email{spectda@gapps.ntnu.edu.tw}

\date{\today}

\begin{document}
	
	\maketitle
	
	\begin{abstract} In this article, we study the fractional spherical maximal function and its lacunary counterpart. We study the necessary and sufficient conditions for $L^p-L^q$ boundedness of both maximal functions. In particular, we prove the restricted weak type estimate for both full and lacunary fractional spherical maximal functions at the boundary of the maximal $L^p-L^q$ bounded regions. 
	\end{abstract}

	\section{Introduction}
	\subsection{Main Results}
	
	For a function $f \in \mathcal{S}(\mathbb{R}^n)$ and $0\leq \alpha\leq n$, we define the fractional spherical average as
	\begin{align}\label{fractional_average}
	A_t^\alpha f(x):=t^{\alpha}\int_{S^{n-1}} f(x-ty) \, d\sigma(y)=t^{\alpha}f\ast\sigma_t(x),
	\end{align}
	where $d\sigma$ is the normalized surface measure on $S^{n-1}$, and consider the fractional spherical maximal function $A_{*}^{\alpha}$ given by
	\begin{align}\label{fmaximal}
		A_{*}^{\alpha}f(x)=\sup_{t>0} \big|A_t^\alpha f(x)\big|.
	\end{align}
	A classical result of Stein \cite{Stein} for $n \geq 3$ and Bourgain \cite{Bourgain} for $n=2$ asserts that for $\alpha=0$, the spherical maximal operator $A_{*}\equiv A_{*}^{0}$ extends as a bounded operator on $L^p(\mathbb{R}^n)$ whenever $n/(n-1)<p \leq \infty$ via the estimate
	\begin{align}\label{SteinBourgain}
		\| A_{*}(f)\|_{L^p(\mathbb{R}^n)} \lesssim \|f\|_{L^p(\mathbb{R}^n)}
	\end{align}
	for all $f \in L^p(\mathbb{R}^n)$.  
	The study of the boundedness properties of $ A_{*}^{\alpha}$ for $\alpha>0$ was initiated by Oberlin \cite{Oberlin} (see also \cite{CGG} for weighted estimates), who showed that a necessary condition for the validity of the correct scaling analogue of \eqref{SteinBourgain}, the inequality
	\begin{align}\label{OberlinSchlagSogge}
		\| A_{*}^{\alpha}(f)\|_{L^{q}(\mathbb{R}^n)} \lesssim \|f\|_{L^p(\mathbb{R}^n)}
	\end{align}
	for all $f \in L^p(\mathbb{R}^n)$, where $q=np/(n-\alpha p)$, is that the point $(1/p,1/q)$ belong to the interior of the region with vertices $O=(0,0), P=(\frac{n-1}{n}, \frac{n-1}{n}), Q=(\frac{n-1}{n}, \frac{1}{n})$ and $R=(\frac{n^2-n}{n^2+1}, \frac{n-1}{n^2+1})$.
	\begin{figure}[H]
		\centering
		\begin{subfigure}[b]{0.45\linewidth}
			\begin{tikzpicture}[scale=4]
				\tiny
				\fill[lightgray] (0,0)--(0.5,0.5)--(2/5,1/5)--cycle;
				\draw[thin][->]  (0,0)node[left]{$O$} --(1.15,0) node[right]{$\frac{1}{p}$};
				\draw[thin][->]  (0,0) --(0,1.2) node[left]{$\frac{1}{q}$};
				\draw[densely dotted] (0,1)node[left]{$(0,1)$}--(1,1)node[right]{$(1,1)$};
				\draw [densely dotted] (1,0)node[below]{$(1,0)$} --(1,1);
				\draw[densely dotted] (2/5,1/5)node[right]{$R$}--(0.5,0.5)node[above]{$P$}--(0,0)--cycle;
				\node at (0.5,0.5) {\tiny{$\circ$}};
				\node at (2/5,1/5) {\tiny{$\circ$}};
				\draw[thin] (0,0) --(0.5,0.5);
			\end{tikzpicture}
			\caption{$n= 2$.}	\label{Fig:full2}
		\end{subfigure}	
		\begin{subfigure}[b]{0.45\linewidth}
			\begin{tikzpicture}[scale=4]
				\tiny
				\fill[lightgray] (0,0)--(0.8,0.8)--(0.8,0.2)--(20/26,4/26)--cycle;
				\draw[thin][->]  (0,0)node[left]{$O$} --(1.15,0) node[right]{$\frac{1}{p}$};
				\draw[thin][->]  (0,0) --(0,1.2) node[left]{$\frac{1}{q}$};
				\draw[densely dotted] (0,1)node[left]{$(0,1)$}--(1,1)node[right]{$(1,1)$};
				\draw [densely dotted] (1,0)node[below]{$(1,0)$} --(1,1);
				\draw[densely dotted] (20/26,4/26)node[below]{$R$}--(0.8,0.2)node[right]{$Q$}--(0.8,0.8)node[above]{$P$}--(0,0)--cycle;
				\node at (0.8,0.8) {\tiny{$\circ$}};
				\node at (4/5,1/5) {\tiny{$\circ$}};
				\node at (20/26,4/26) {\tiny{$\circ$}};
				\draw[thin] (0,0) --(0.8,0.8);
			\end{tikzpicture}
			\caption{$n\geq 3$.}
		\end{subfigure}	
		\caption{Oberlin's necessary conditions for $L^p(\R^n)\to L^q(\R^n)$ boundedness of $A_{*}^{\alpha}$.}\label{Fig:full}
	\end{figure}
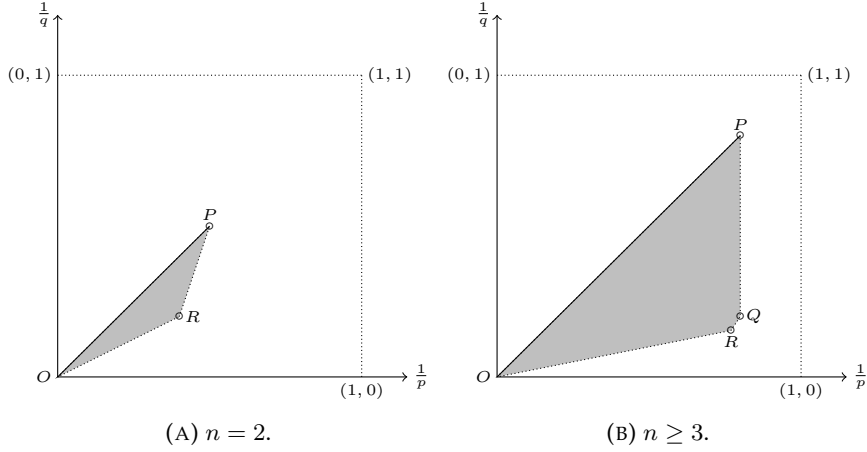

	Besides giving necessary conditions, Oberlin also showed that for $n \geq 3$ a sufficient condition for the validity of the inequality \eqref{OberlinSchlagSogge} is that the point $(1/p,1/q)$ belong to the interior of the triangle with vertices  $O=(0,0), P=(\frac{n-1}{n}, \frac{n-1}{n}), Q=(\frac{n-1}{n}, \frac{1}{n})$.  As explained by Schlag  \cite[p.~106]{Schlag} (see also \cite[Theorem 3.1.1 on p.~50]{SchlagThesis}),  and noted by Kinnunen and Saksman \cite[p.~534]{KS}, Beltran, Ramos, and Saari \cite[p.~1878]{BRS}, by an argument due to Bourgain \cite[p.~72-74]{Bourgain1986}, estimates for the local spherical maximal function 
	\begin{align}\label{local_spherical}
		\widetilde{A}_{*}f(x)=\sup_{t \in [1,2]} \big|A_t f(x)\big|
	\end{align}
	obtained by Schlag \cite{Schlag}, Schlag and Sogge \cite{Schlag-Sogge} imply the validity of the inequality \eqref{OberlinSchlagSogge} for $(1/p,1/q)$ in the interior of the region $OPQR$ when $n \geq 3$ and the triangle $OPR$ when $n=2$.  As the boundary $OP$ corresponds to the estimates of Bourgain and Stein, the only remaining question concern the boundaries $PQ$, $QR$, and $OR$ for $n\geq 3$ and $PR$ and $OR$ for $n=2$.

	The first result of this paper is the following restricted weak-type estimate for the open segments $PQ$, $QR$, and $OR$ for $n \geq 3$:
	\begin{theorem}\label{thm:full}
		Let $n\geq 3$, $\alpha \in (0,n)$, and $\frac{n}{n-1}\leq p<q\leq\infty$ with $\frac{\alpha}{n}=\frac{1}{p}-\frac{1}{q}$. If $\big(\frac{1}{p},\frac{1}{q}\big)$ on the open line segment $OR$, $QR$ or $PQ$, we have
		\begin{align*}
			\|A^{\alpha}_{*}f\|_{L^{q,\infty}(\mathbb{R}^n)}\lesssim \|f\|_{L^{p,1}(\mathbb{R}^n)}.
		\end{align*}
	\end{theorem}
	
	When $n=2$ we have similar estimates for the open segments  $PR$ and $OR$:
	\begin{theorem}\label{thm:full2}
		Let $\alpha \in (0,2)$, $2< p<q\leq\infty$ with $\frac{\alpha}{2}=\frac{1}{p}-\frac{1}{q}$. If $\big(\frac{1}{p},\frac{1}{q}\big)$ lies on the open line segment $PR$ or $RO$, we have
		\begin{align*}
			\|A^{\alpha}_{*}f\|_{L^{q,\infty}(\mathbb{R}^2)}\lesssim \|f\|_{L^{p,1}(\mathbb{R}^2)}.
		\end{align*}
	\end{theorem}

	By standard interpolation, Theorems \ref{thm:full} and \ref{thm:full2} recover the consequences of Schlag's \cite{Schlag}, Schlag and Sogge's \cite{Schlag-Sogge} results noted by Kinnunen and Saksman \cite[p.~534]{KS}:    
	\begin{theorem}\label{standard_interpolation}
		Let $n\geq 2$, $\alpha \in [0,n]$, and $\frac{n}{n-1}\leq p\leq q\leq\infty$ with $\frac{\alpha}{n}=\frac{1}{p}-\frac{1}{q}$.  If $\big(\frac{1}{p},\frac{1}{q}\big)$ lies in the interior of the region $OPQR$ or the open line segment $OP$, then
		\begin{align*}
			\|A_{*}^{\alpha}f\|_{L^q(\mathbb{R}^n)}\leq C \|f\|_{L^p(\mathbb{R}^n)}.
		\end{align*}
	\end{theorem}
		
		We next consider the lacunary fractional spherical maximal operator
		\begin{align}\label{lacunary_fmaximal}
			A_{lac}^{\alpha}f(x)=\sup_{k\in \mathbb{Z}} \big|A_{2^k}^\alpha f(x)\big|.
		\end{align}
		For $1<p \leq \infty$, Coifman and Weiss \cite[p.~246]{CoifmanWeiss} and C.P. Calder\'on \cite{Calderon_lacunary} (see also \cite{DRF}) proved an analogue of Bourgain and Stein's inequality \eqref{SteinBourgain}, that $A_{lac}\equiv A_{lac}^{0}$ operator admits a bounded extension on $L^p(\mathbb{R}^n)$:  For $n \geq 2$, one has
		\begin{align}\label{Calderon}
			\| A_{lac}(f)\|_{L^p(\mathbb{R}^n)} \lesssim \|f\|_{L^p(\mathbb{R}^n)}
		\end{align}
		for all $f \in L^p(\mathbb{R}^n)$.  When $p=1$,  Christ \cite{Christ} established an estimate from the Hardy space $\mathcal{H}^1(\mathbb{R}^n)$ to weak $L^1(\mathbb{R}^n)$:
		\begin{align}
			\sup_{t>0} \left|\left\{A_{\text{lac}}(f)>t\right\}\right|  &\lesssim \|f\|_{\mathcal{H}^1(\mathbb{R}^n)}
		\end{align}
		for all $f \in \mathcal{H}^1(\mathbb{R}^n)$.   The question of whether $A_{\text{lac}}$ maps $L^1(\mathbb{R}^n)$ to weak $L^1(\mathbb{R}^n)$ is still open, we refer the reader to \cite{STW,CladekKrause}, where estimates in Orlicz spaces near $L^1(\mathbb{R}^n)$ into weak $L^1(\mathbb{R}^n)$ are established, and to the survey \cite{RoosSeeger} for more details.
		
		Turning our attention to the case $\alpha>0$, we first observe that a necessary condition for the boundedness of the lacunary fractional spherical maximal operator is the boundedness of the averaging operator.   In combination with Lacey's examples \cite[Proposition 5.1(1) on p.~624]{Lacey} this implies that 
		\begin{align}
			\|A^{\alpha}_{lac}f\|_{L^{q}(\mathbb{R}^n)}\lesssim \|f\|_{L^{p}(\mathbb{R}^n)}
		\end{align}
		can only hold for all $f \in L^p(\mathbb{R}^n)$ if $(1/p,1/q)$ belongs to the triangle with vertices $O=(0,0), A=(1,1),$ and  $B=(\frac{n}{n+1}, \frac{1}{n+1})$:
		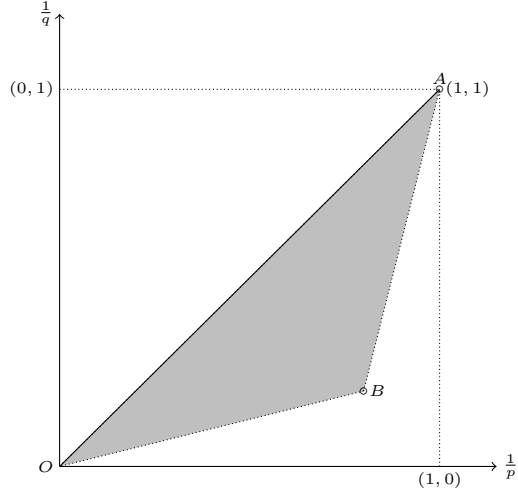
\begin{figure}[H]
			\centering
			\begin{tikzpicture}[scale=5]
				\tiny
				\fill[lightgray] (0,0)--(0.8,0.2)--(1,1)--cycle;
				\draw[thin][->]  (0,0)node[left]{$O$} --(1.15,0) node[right]{$\frac{1}{p}$};
				\draw[thin][->]  (0,0) --(0,1.2) node[left]{$\frac{1}{q}$};
				\draw[densely dotted] (0,1)node[left]{$(0,1)$}--(1,1)node[right]{$(1,1)$};
				\draw [densely dotted] (1,0)node[below]{$(1,0)$} --(1,1);
				\draw[densely dotted] (0.8,0.2)node[right]{$B$}--(1,1)node[above]{$A$}--(0,0)--cycle;
				\node at (1,1) {\tiny{$\circ$}};
				\node at (0.8,0.2) {\tiny{$\circ$}};
				\draw[thin] (0,0) --(1,1);
			\end{tikzpicture}
			\caption{Necessary conditions for $L^p(\R^n)\to L^q(\R^n)$ boundedness of $A_{lac}^{\alpha}$.}
			\label{Fig:lac}
		\end{figure}
		
		Our next result shows that the estimate holds in the interior of this region, with restricted weak-type estimates on the open segments of the boundary.
		\begin{theorem}\label{thm:lac}
			Let $n\geq 2$, $\alpha \in [0,n]$, and $1<p\leq q\leq \infty$ with $\frac{\alpha}{n}=\frac{1}{p}-\frac{1}{q}$. If $\big(\frac{1}{p},\frac{1}{q}\big)$ lies in the interior of the triangle $OAB$ or the open line segment $OA$, then
			\begin{align*}
				\|A^{\alpha}_{lac}f\|_{L^q(\mathbb{R}^n)}\lesssim \|f\|_{L^p(\mathbb{R}^n)}.
			\end{align*}
			Moreover, when $\big(\frac{1}{p},\frac{1}{q}\big)$ lies on the open line segment $OB$ or $AB$, we have
			\begin{align*}
				\|A^{\alpha}_{lac}f\|_{L^{q,\infty}(\mathbb{R}^n)}\lesssim \|f\|_{L^{p,1}(\mathbb{R}^n)}.
			\end{align*}
		\end{theorem}

		Our results establish restricted weak-type estimates on the portions of the boundary (up to the endpoints) that are not ruled out by Stein's counterexample.  This may prompt one to wonder whether weak-type estimates or even strong-type estimates hold on the boundary, save the segment $PQ$ in Figure \ref{Fig:full} (where Stein's counterexample shows no improvement can be made).  For the lacunary case, on the segment $AB$, no strong-type estimate is possible, as the following example shows.
		\begin{example}
			Let $f(x)=\chi_{B(0,10)}(x)$, then $f\in L^p(\R^n)$ for all $p\geq1$. Fix $l\geq100$, then for $2^l-1\leq |x|\leq 2^l+1$, we have
			\begin{align*}
			A^{\alpha}_{lac}f(x)\gtrsim2^{-l(n-1-\alpha)}.
			\end{align*}
			Thus,
			\begin{align*}
				\int_{\R^n}|A^{\alpha}_{lac}f(x)|^q\;dx\gtrsim \sum_{l\geq100}2^{-l(n-1-\alpha)q}2^{l(n-1)}=2^{l(n-1-(n-1-\alpha)q)}
			\end{align*}
			If $A^{\alpha}_{lac}f\in L^q(\R^n)$ then $(n-1-\alpha)q>n-1$. This shows that $L^p(\R^n)\to L^q(\R^n)-$ bound does not hold on the line segment $AB$.
		\end{example}	
A more careful application of a complex interpolation result of Sagher yields a slight improvement on the segments $OR$, $OB$, that $A^{\alpha}_{*},A^{\alpha}_{lac}$ map $L^{p,r}$ to $L^{q,\infty}$ for some $r=r(p,q)>1$.  However, this falls short of the weak-type estimate and it remains an interesting open problem whether the estimates can be further improved on the remaining portions of the boundary, though this seems a hard question.  In particular, even weak-type boundedness of the fractional spherical maximal operator would imply weak-type bounds for the local spherical maximal function.  By interpolation, this in turn would imply strong-type bounds away from the corners, estimates which have been established by Lee in \cite{Lee1}, where the argument is highly non-trivial.
		
		One can contrast the type set for the fractional spherical and lacunary fractional spherical maximal functions with that of the fractional maximal function, defined for $\alpha \in [0,n]$ by
		\begin{align}\label{fractionalmaximal}
			\mathcal{M}^\alpha f(x)= \sup_{r>0} r^\alpha \fint_{B(x,r)} |f(y)|\;dy,
		\end{align}
		and whose full range of validity of estimates are depicted in the following figure.
		\begin{figure}[H]
			\centering
			\begin{tikzpicture}[scale=5]
				\tiny
				\fill[lightgray] (0,0)--(1,0)--(1,1)--cycle;
				\draw[thin][->]  (0,0)node[left]{$O$} --(1.15,0) node[right]{$\frac{1}{p}$};
				\draw[thin][->]  (0,0) --(0,1.2) node[left]{$\frac{1}{q}$};
				\draw[densely dotted] (0,1)node[left]{$(0,1)$}--(1,1)node[right]{$(1,1)$};
				\draw [ dotted] (1,0)node[below]{$(1,0)$} --(1,1);
				\draw[densely dotted] (1,0)--(1,1)--(0,0)--cycle;
				\draw[thin] (0,0) --(1,1);
			\end{tikzpicture}
			\caption{The figure denotes the regions of $L^p(\R^n)\to L^q(\R^n)$ boundedness of $\mathcal{M}^{\alpha}$.}
			\label{Fig:fractional}
		\end{figure}
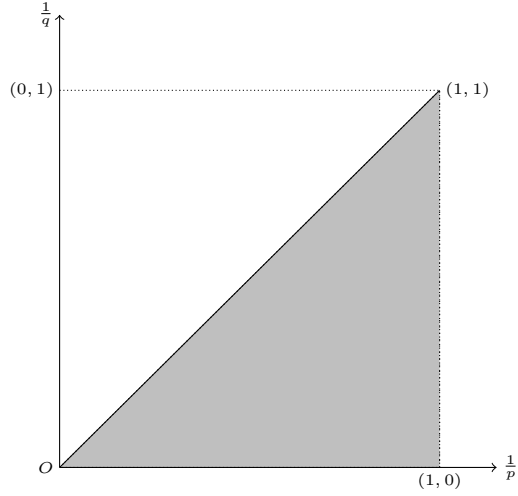
		\noindent
		In particular, $\mathcal{M}^\alpha$ maps $L^p(\R^n)\to L^q(\R^n)$ in the interior of the triangle, $L^1(\R^n)$ to weak-$L^{n/(n-\alpha)}(\R^n)$ on the right edge with failure of the strong-type estimate,  and $L^{n/\alpha}(\R^n)$ (and even weak-$L^{n/ \alpha}(\R^n)$) into $L^\infty(\R^n)$.  The comparison is only cursory, as the two operators are vastly different in character.  The introduction of the fractional maximal operator will, however, be useful in the sequel in the proofs.

		The plan of the paper is as follows.  In Section \ref{preliminaries} we introduce some requisite preliminaries.  In Section \ref{lacunary}, the prove Theorem \ref{thm:lac}.  In Section \ref{full}, we prove Theorem \ref{thm:full}.  In Section \ref{full2} we prove Theorem \ref{thm:full2}.
		
		\section{Preliminaries}\label{preliminaries}
		Let $\phi, \psi \in \mathscr{S}(\mathbb{R}^n)$ be a functions such that $\supp(\widehat{\phi})\subset B(0,1)$ and $\widehat{\psi}(\xi)=1$ on $B(0,2)\setminus B(0,1/2)$ with
		\begin{align}\label{identity}
			\widehat{\phi}(\xi)
			+ \sum_{j\geq1} \widehat{\psi}_{2^j}(\xi)=1, \quad \xi\neq 0.
		\end{align}
		where $\widehat{\psi}_t(\xi)=\widehat{\psi}\big(\frac{\xi}{t}\big)$. 
		
		Using the partition of unity above, we can decompose $A_{lac}^{\alpha}$ as follows.
		\begin{align*}
			A_{lac}^{\alpha}f(x)&=\sup_{k\in\Z}2^{k\alpha}\big|f\ast\phi_{2^{-k}}\ast\sigma_{2^k}(x)\big|+\sum_{j\geq1}\sup_{k\in\Z}2^{k\alpha}\big|f\ast\psi_{2^{j-k}}\ast\sigma_{2^k}(x)\big|\\
			&=\sum_{j\geq0}M_j^{\alpha}f(x),
		\end{align*}
		where
			\begin{align}\label{lacunary_LP_maximal_0}
	M_0^{\alpha}f(x)=\sup_{k\in\Z}2^{k\alpha}\big|f\ast\phi_{2^{-k}}\ast\sigma_{2^k}(x)\big|,
	\end{align}
and for $j\geq1$
		\begin{align}\label{lacunary_LP_maximal}
		M_j^{\alpha}f(x)=\sup_{k\in\Z}2^{k\alpha}\big|f\ast\psi_{2^{j-k}}\ast\sigma_{2^k}(x)\big|.
		\end{align}

Similarly, we have the following decomposition of $A_{*}^{\alpha}$.
		\begin{align*}
			A_{*}^{\alpha}f(x)&=\sup_{k\in\Z}\sup_{t\in[2^{k},2^{k+1})}\big|A_{t}^{\alpha}f(x)\big|\\
			&\leq\sup_{k\in\Z}\sup_{t\in[2^{k},2^{k+1})}t^{\alpha}\big|f\ast\phi_{2^k}\ast\sigma_{t}(x)\big|+\sum_{j\geq1}\sup_{k\in\Z}\sup_{t\in[2^{k},2^{k+1})}t^{\alpha}\big|f\ast\psi_{2^{j-k}}\ast\sigma_{t}(x)\big|\\
			&=:\sum_{j\geq0}M_{*,j}^{\alpha}f(x),
		\end{align*}
		where 
		\begin{align}\label{LP_maximal_0}
		M_{*,0}^{\alpha}f(x)=\sup_{k\in\Z}\sup_{t\in[2^{k},2^{k+1})}t^{\alpha}\big|f\ast\phi_{2^k}\ast\sigma_{t}(x)\big|,
		\end{align}
		and for $j\geq1$
\begin{align}\label{LP_maximal}              
M_{*,j}^{\alpha}f(x)&=\sup_{k\in\Z}M_{k,j}^{\alpha}f(x)=\sup_{k\in\Z}\sup_{t\in[2^{k},2^{k+1})}t^{\alpha}\big|f\ast\psi_{2^{j-k}}\ast\sigma_{t}(x)\big|.
\end{align}

Note that for $j\geq 1$, we can write  
         \begin{align}\label{max-equiv}
M_{*,j}^{\alpha}g(x)=\sup_{k\in\Z}\sup_{t\in[2^{k},2^{k+1})}t^{\alpha}\big|g\ast\psi_{2^{j-k}}\ast\sigma_{t}(x)\big|
= \sup_{t>0} t^{\alpha} \left|g\ast \psi_{2^{j}a(t)}\ast \sigma_{t}(x)\right|,
	\end{align}

where $a:[0,\infty)\rightarrow [0,\infty)$ is a piecewise continuous function defined by 
\begin{align*}
 a(t)= \sum_{k\in \mathbb{Z}} 2^{-k} \chi_{[2^k, 2^{k+1})(t)}.
 \end{align*}
		
		For complex interpolation of an analytic family of operators, it is useful to construct linearizations of the maximal functions \eqref{fmaximal}, \eqref{lacunary_fmaximal}, \eqref{lacunary_LP_maximal}, and \eqref{LP_maximal}.  We give a sketch of the proof for \eqref{LP_maximal}, following the argument in \cite[Lemma 3.3]{BasakSpector}.  The others can be argued analogously.
		\begin{lemma}\label{pointwise_approximation}
			Let $j\geq 1$ and $\alpha \in [0,n]$. Let $g$ be a measurable function such that the maximal function satisfies
						\begin{align*}
M_{*,j}^{\alpha}g(x)=\sup_{k\in\Z}\sup_{t\in[2^{k},2^{k+1})}t^{\alpha}\big|g\ast\psi_{2^{j-k}}\ast\sigma_{t}(x)\big|<+\infty.
			\end{align*}
for almost every $x \in \mathbb{R}^n$.  For each $\ell \in \mathbb{N}$, there exists a measurable function $t_\ell:\mathbb{R}^n \to [1/\ell,\ell]\cap \mathbb{Q}$ and a measurable function $\eta_\ell :\mathbb{R}^n \to \mathbb{R}$ such that for every integrable function $h$, the function
\begin{align}\label{measurable_approximation}
				U_{j,\ell}(h)(x):= \eta_\ell(x) t_{\ell}^{\alpha} \cdot  h\ast \psi_{2^j a(t_{\ell})} \ast \sigma_{t_{\ell}} 
			\end{align}
is measurable, satisfies the upper bound
\begin{align}\label{linearization_upper_bound}
				|U_{j,\ell}(h)(x)| \leq M_{*,j}^{\alpha}h(x),
			\end{align}
and
\begin{align}\label{convergence}
				M_{*,j}^{\alpha}g(x) = \lim_{\ell \to \infty} U_{j,\ell}(g)(x) 
			\end{align}
			for pointwise almost every $x \in \mathbb{R}^n$.
		\end{lemma}
        \begin{proof}
         First, using \eqref{max-equiv}, we write

         \begin{align*}
M_{*,j}^{\alpha}g(x)=\sup_{t>0} t^{\alpha} \left|g\ast \psi_{2^{j}a(t)}\ast \sigma_{t}(x)\right|.
         \end{align*}
For $\ell \in \mathbb{N}$, we define $\Lambda_{\ell}=[1/\ell, \ell]\cap \mathbb{Q}$ and the truncated maximal function 

\begin{align*}
   M_{*,j,\ell}^{\alpha} g(x)=\sup_{t\in \Lambda_{\ell}} t^{\alpha} \left|g\ast \psi_{2^{j}a(t)}\ast \sigma_{t}(x)\right|. 
\end{align*}

Let $\{t^{i}_{\ell}\}$ be an enumeration of $\Lambda_{\ell}$. We define $t_{\ell}(x)=t_{\ell}^{1}$ if 
\begin{align*}
 (t_{\ell}^{1})^{\alpha} \left| g\ast \psi_{2^j a(t_{\ell}^1)} \ast \sigma_{t_{\ell}^1}\right| > \left(1-\frac{1}{\ell} \right)   M_{*,j,\ell}^{\alpha} g(x).
\end{align*}

It is easy to see that the set of all such $x$ for which $t_{\ell}(x)=t_{\ell}^{1}$ is measurable and we denote $\Omega_{\ell}^{1}$ be the set of all such $x$ selected in this step.

Next choose all those $x\notin \Omega_{\ell}^{1}$ such that 
\begin{align*}
 (t_{\ell}^{2})^{\alpha} \left| g\ast \psi_{2^j a(t_{\ell}^2)} \ast \sigma_{t_{\ell}^2}\right| > \left(1-\frac{1}{\ell} \right)   M_{*,j,\ell}^{\alpha} g(x),
\end{align*}
and inductively, we define $t_{\ell}^{k}$ be the smallest integer $k$ such that 
\begin{align}\label{induction}
    (t_{\ell}^{k})^{\alpha} \left| g\ast \psi_{2^j a(t_{\ell}^k)} \ast \sigma_{t_{\ell}^k}\right| > \left(1-\frac{1}{\ell} \right)   M_{*,j,\ell}^{\alpha} g(x).
\end{align}
Further define $\Omega_{\ell}^{k}$ be all such $x$ selected in this step. 

The function $t_{\ell}(x)$ can be expressed as
\begin{align*}
    t_{\ell}(x)=\sum_{k=1}^{\infty} t_{\ell}^{k} \chi_{\Omega_{\ell}^{k}}(x),
\end{align*}
and in particular, as it take on only countably many values, it is a measurable function of $x$. 

The measurability of the function $x\mapsto g\ast \psi_{2^j a(t_{\ell})} \ast \sigma_{t_{\ell}}$ implies measurability of the function 
 \begin{align*}
     \eta_{\ell}(x)= \left\{
     \begin{array}{ll}
           \frac{g\ast \psi_{2^j a(t_{\ell})} \ast \sigma_{t_{\ell}}}{\left| g\ast \psi_{2^j a(t_{\ell})} \ast \sigma_{t_{\ell}}\right|} & \quad \text{if} \left| g\ast \psi_{2^j a(t_{\ell})} \ast \sigma_{t_{\ell}}\right|\neq 0,\\
           0 &\quad \text{otherwise}.\\
     \end{array}
     \right.
\end{align*}
Therefore, the function defined in \eqref{measurable_approximation} is a measurable function.

The upper bound \eqref{linearization_upper_bound} follows easily from the definition of $U_{j,\ell}$, while the convergence \eqref{convergence} follows from the fact 
\begin{align*}
      M_{*,j}^{\alpha} g(x)= \lim_{\ell \rightarrow \infty}   M_{*,j,\ell}^{\alpha} g(x)
\end{align*}
and \eqref{induction}.

        \end{proof}
		
	We will make liberal use of complex interpolation in this paper, including standard interpolations between Lebesgue spaces, as well as complex interpolation with a restricted weak-type estimate.  A convenient statement of hypothesis that includes these various scenarios can be found in the paper of Sagher \cite[p.~351]{Sagher}, whose result we next recall.
		\begin{lemma}\label{complexinter}
			Let $S$ denote the closed strip $\{z\in \mathbb{C}: 0\leq Re(z)\leq 1\}$ and suppose $\{T_z\}_{z \in S}$ is an analytic family of linear operators such that
			\begin{align*}
				\|T_{it} f\|_{L^{q_0,s_0}(\mathbb{R}^n)} &\leq C_0(t)\|f\|_{L^{p_0,r_0}(\mathbb{R}^n)},\\
				\|T_{1+it} f\|_{L^{q_1,s_1}(\mathbb{R}^n)} &\leq C_1(t)\|f\|_{L^{p_1,r_1}(\mathbb{R}^n)},
			\end{align*}
for all simple functions $f$, where $\log C_i(t) \leq C exp(a|t|)$ for some $a<\pi$.  Then for $\theta \in (0,1)$,
			\begin{align*}
				\frac{1}{q} &= \frac{\theta}{q_0} + \frac{1-\theta}{q_1}, \qquad \frac{1}{s} = \frac{\theta}{s_0} + \frac{1-\theta}{s_1}, \\
				\frac{1}{p} &= \frac{\theta}{p_0} + \frac{1-\theta}{p_1}, \qquad \frac{1}{r} = \frac{\theta}{r_0}+ \frac{1-\theta}{r_1},
\end{align*}
we have for all simple functions
\begin{align*}
\|T_{\theta} f\|_{L^{q,s}(\mathbb{R}^n)} \leq C_\theta\|f\|_{L^{p,r}(\mathbb{R}^n)}
\end{align*}
where 
\begin{align*}
\log C_\theta \lesssim \int_{-\infty}^\infty \frac{\log C_0(t)}{\cosh(\pi t)- \cos(\pi \theta)} \;dt +\int_{-\infty}^\infty \frac{\log C_1(t) }{\cosh(\pi t)+ \cos(\pi \theta)}\;dt.
\end{align*}
\end{lemma}
\noindent
In fact, the underlying ideas in the more standard reference \cite[Theorem 4.1 on p.~205]{SteinWeiss} are sufficient for our purposes.  However, in this case one should observe the interpolation of a restricted weak-type estimate with an estimate into the space of bounded functions yields an interpolated restricted weak-type estimate.	
		
In addition to complex interpolation, we also make use of Bourgain's interpolation argument to prove the restricted weak-type estimates for $A_{lac}^{\alpha} .$  We state the interpolation lemma for convenience. The reader is referred to \cite{Lee1} for more details. 
		\begin{lemma}[\cite{Lee1}, Lemma 2.6]\label{Bourgain}
			Let $\epsilon_1,\epsilon_2>0$. Suppose that $\{T_j\}$ is a sequence of linear (or sublinear) operators 
			such that for some $1\leq p_1,p_2<\infty$, and $1\leq q_1,q_2<\infty$, 
			$$\Vert T_{j}(f)\Vert_{L^{q_1}(\R^n)}\leq M_12^{\epsilon_1 j}\Vert f\Vert_{L^{p_1}(\R^n)},~~\Vert T_{j}(f)\Vert_{L^{q_2}(\R^n)}\leq M_22^{-\epsilon_2 j}\Vert f\Vert_{L^{p_2}(\R^n)}.$$
			Then $T=\sum_jT_j$ is bounded from $L^{p,1}(\R^n)$ to $L^{q,\infty}(\R^n)$ with 
			$$\Vert T(f)\Vert_{L^{q,\infty}(\R^n)}\lesssim M^{\theta}_{1}M^{1-\theta}_{2}\Vert f\Vert_{L^{p,1}(\R^n)},$$	
			where $\theta=\frac{\epsilon_2}{\epsilon_1+\epsilon_2}$, $\frac{1}{q}=\frac{\theta}{q_1}+\frac{1-\theta}{q_2}$ 
			and $\frac{1}{p}=\frac{\theta}{p_1}+\frac{1-\theta}{p_2}$.
		\end{lemma}

		\section{Proof of Theorem \ref{thm:lac}}\label{lacunary}

		We begin with several lemmas for the operators $M_j^{\alpha}$ defined in \eqref{lacunary_LP_maximal_0} and \eqref{lacunary_LP_maximal}.
		\begin{lemma}\label{Pointwise:lac}
			Let $j\geq0$ and $0\leq\alpha\leq n$, then we have
			\[M_j^{\alpha}f(x)\lesssim 2^j\mathcal{M}^{\alpha}f(x).\]
		\end{lemma}
		\begin{proof}
			For each $j$ and $k$, we have the following kernel estimate
			\[|\psi_{2^{j-k}}\ast\sigma_{2^k}(y)|\lesssim2^j2^{-kn}(1+2^{-k}|y|)^{-N},\quad\text{for }\quad N>n.\]
			Then,
			\begin{align*}
				&2^{k\alpha}|f\ast\psi_{2^{j-k}}\ast\sigma_{2^k}(x)|\\
				\lesssim& 2^j2^{k(\alpha-n)}\int_{\R^n}f(x-y)(1+2^{-k}|y|)^{-N}\;dy\\
				\leq&2^j2^{k(\alpha-n)}\int_{|y|\leq 2^{k}}f(x-y)\;dy+\sum_{l\geq 1}2^j2^{k(\alpha-n)}2^{-lN}\int_{2^{k+l-1}<|y|\leq 2^{k+l}}f(x-y)\;dy\\
				\lesssim&2^j\mathcal{M}^{\alpha} f(x)\bigg(1+\sum_{l\geq1}2^{-l(N-n+\alpha)}\bigg)\lesssim 2^j\mathcal{M}^{\alpha} f(x).
			\end{align*}
			Since, the above estimate is uniform in $k$, we obtain the desired bound
			\[M_j^{\alpha}f(x)\lesssim 2^j\mathcal{M}^{\alpha}f(x).\]        
		\end{proof}

		\begin{lemma}\label{decay:lac}
			Let $j\geq 1$ and $1\leq p\leq\infty$, then we have
			\[\|M_j^0f\|_{L^p(\R^n)}\lesssim2^{-j\frac{n-1}{\max\{p,p'\}}}\|f\|_{L^p(\R^n)}.\]
		\end{lemma}
		
		\begin{proof}
			From the $\ell_2\hookrightarrow\ell_\infty$, we have
			\[M_j^{0}f(x)\leq \bigg(\sum_{k\in\Z}|f\ast\psi_{2^{j-k}}\ast\sigma_{2^k}(x)|^2\bigg)^{\frac{1}{2}}.\]
			Using the Plancherel theorem and $|\widehat{\sigma_{2^k}}(\xi)|\lesssim(1+2^k|\xi|)^{-\frac{n-1}{2}}$, we obtain that
			\begin{align*}
				\|M_j^{0}f\|_{L^2(\R^n)}&=\bigg(\sum_{k\in\Z}\|\widehat{f}\widehat{\psi}_{2^{j-k}}\widehat{\sigma}_{2^k}\|_{L^2(\R^n)}^2\bigg)^{\frac{1}{2}}\\
				&\lesssim 2^{-j\frac{n-1}{2}}\bigg(\sum_{k\in\Z}\|\widehat{f}\widehat{\psi}_{2^{j-k}}\|_{L^2(\R^n)}^2\bigg)^{\frac{1}{2}}\\
				&\lesssim 2^{-j\frac{n-1}{2}}\|f\|_{L^2(\R^n)}.
			\end{align*}
			This proves the desired bounds for $p=2$. When $p>2$, interpolating between the above $L^2(\R^n)$-estimate and trivial $L^\infty(\R^n)$ bound for lacunary spherical maximal function gives the desired decay for $L^p(\R^n)$-estimate. Now, we will use the bootstraping argument when $1\leq p<2$. Consider the following vector-valued operator $\vec{\bold{A}}_{j}$ acting on a sequence of measurable functions $\bold{f}=\{f*\psi_{k-j}\}_{k\in\Z}$,
			\[\vec{\bold{A}}_{j}\bold{f}=\{A_{2^k,j}^{0}(f*\psi_{k-j})\}_{k\in\Z},\]
			where $A_{2^k,j}^{0}f(x)=A_{2^k}^{0}(f\ast\psi_{k-j})(x)$.
			From the $L^1(\R^n)$-boundedness of $A_{2^k}^{0}$ and finite $L^1(\R^n)$ norm of $\psi_{k-j}$, we have
			\[\|\vec{\bold{A}}_{j}\|_{L^1(\R^n,\ell_1)\to L^1(\R^n,\ell_1)}\lesssim1.\]
			Next, using the $L^2$-boundedness of $M_{j}^0$, we get
			\begin{align*}
				\|\vec{\bold{A}}_{j}\bold{f}\|_{L^2(\R^n,\ell_\infty)}
				&= \left\|\sup_{k\in\Z}|A_{2^k,j}^{0}(\bold{f})|\right\|_{L^2(\R^n)}\\&\leq \left\|M_{j}^0(\sup_{m}|f*\psi_{m-j}|)\right\|_{L^2(\R^n)}\\
				&\lesssim {2^{-j{{n-1}\over2}}}\|\bold{f}\|_{L^2(\R^n,\ell_\infty)}.
			\end{align*}
			Thus, interpolating ${L^1(\R^n,\ell_1)}$ and ${L^2(\R^n,\ell_\infty)}$ bounds of $\vec{\bold{A}}_{j}$, we get that
			\[\|\vec{\bold{A}}_{j}\bold{f}\|_{L^{4\over3}(\R^n,\ell_2)}\lesssim{2^{-j{{n-1}\over4}}}\|\bold{f}\|_{L^{4\over3}(\R^n,\ell_2)}.\]
			Again using the $\ell_{2}\hookrightarrow\ell_{\infty}$ embedding and above estimate, we have
			\begin{align*}
				\|M_{j}^0f\|_{L^{4\over3}(\R^n)}&\leq \left\|\left(\sum_{k}|A_{j}^{k}f|^2\right)^{1\over2}\right\|_{L^{4\over3}(\R^n)}\\
				&\lesssim {2^{-j{{n-1}\over4}}}\left\|\left(\sum_{k}|f*\psi_{k-j}|^2\right)^{1\over2}\right\|_{L^{4\over3}(\R^n)}\\
				&\lesssim {2^{-j{{n-1}\over4}}}\|f\|_{L^{4\over3}(\R^n)},
			\end{align*}
			where we used Littlewood--Paley inequality in the last inequality. Using the similar argument inductively, we get
			\begin{align*}
				\|M_{j}^0f\|_{L^p(\R^n)}\lesssim {2^{-j\frac{n-1}{p'}}}\|f\|_{L^p(\R^n)},
			\end{align*}
			for all $1<p\leq2$.
		\end{proof}
		
		\begin{lemma}\label{dual:lac}
			Let $j\geq 1$ and $0\leq\gamma\leq n$, then we have
			\[\|M_j^{\gamma}f\|_{L^{\frac{2n}{n-\gamma}}(\R^n)}\lesssim2^{-j\big(\frac{(n-1)(n-\gamma)}{2n}-\frac{\gamma}{n}\big)}\|f\|_{L^{\frac{2n}{n+\gamma}}(\R^n)}.\]
		\end{lemma}
		\begin{proof}
			Note that when $\gamma=n$, from Lemma \ref{Pointwise:lac}, we have
			\begin{equation}\label{alpha=n}
				\|M_j^nf\|_{L^\infty(\R^n)}\lesssim2^j\|f\|_{L^1(\R^n)}.
			\end{equation}
			For $0<\gamma<n$, using a complex interpolation between the estimate \eqref{alpha=n} and the estimate in Lemma \ref{decay:lac} with $p=2$, we obtain that
			\[\|M_j^{\gamma}f\|_{L^{\frac{2n}{n-\gamma}}(\R^n)}\lesssim2^{-j\big(\frac{(n-1)(n-\gamma)}{2n}-\frac{\gamma}{n}\big)}\|f\|_{L^{\frac{2n}{n+\gamma}}(\R^n)}.\]
		\end{proof}

		\begin{proof}[Proof of Theorem \ref{thm:lac}]
			
			Note that when $j=0$, the boundedness region of Theorem \ref{thm:lac} will be contained in the boundedness region of $\mathcal{M}^{\alpha}$. Therefore, it is enough to consider the sum $\sum_{j\geq1}M_j^{\alpha}f$.

			Using complex interpolation between the $L^p$-estimates in Lemma \ref{decay:lac} and \eqref{alpha=n}, we have
			\[\|M_j^{\alpha}f\|_{L^q(\R^n)}\lesssim2^{-j\epsilon}\|f\|_{L^p(\R^n)},\quad\text{ for some }\epsilon=\epsilon(p,q)>0,\]
			whenever $\big(\frac{1}{p},\frac{1}{q}\big)$ lies in the triangle $OAB$. Therefore it remains to prove restricted weak-type estimates on the open line segments $OB$ and $AB$.

			\subsection*{Open line segment $OB$.}
			Fix $0<\alpha<\frac{n(n-1)}{n+1}$.  If $\big(\frac{1}{p},\frac{1}{q}\big)\in OB$, then $p=\frac{n-1}{\alpha}$ and $q=\frac{n(n-1)}{\alpha}$.  We can see that any such point on the open line segment $OB$ lies on a line with unit slope connecting two new line segments introduced in the following figure. 
				\begin{figure}[H]
			\centering
			\begin{tikzpicture}[scale=7]
				\tiny
				\draw[thin][->]  (0,0)node[left]{$O$} --(1.15,0) node[right]{$\frac{1}{p}$};
				\draw[thin][->]  (0,0) --(0,1.2) node[left]{$\frac{1}{q}$};
				\draw[densely dotted] (0,1)node[left]{$(0,1)$}--(1,1)node[right]{$(1,1)$};
				\draw [densely dotted] (1,0)node[below]{$(1,0)$} --(1,1);
				\draw[thin] (0.8,0.2)node[right]{$B$}--(1,1)node[above]{$A$}--(0,0)--cycle;
				\draw[densely dotted] (0.5,0.5)node[right]{$B_0$}--(1,0);
				\draw[densely dotted] (0,0)--(0.9,0.1)node[right]{$B_1$};
				\draw[densely dotted] (1,1)--(0.9,0.1);
				\draw[thin] (3.15/8,0.35/8)--(.675,0.325);
				\node at (1.4/3,0.35/3) {\tiny{$\circ$}};
				\node at (1.1/3,0.54/3) {\tiny{$\left(\frac{\alpha}{n-1},\frac{\alpha}{n(n-1)}\right)$}};

				\node at (1,1) {\tiny{$\circ$}};
				\node at (0.8,0.2) {\tiny{$\circ$}};
				\draw[thin] (0,0) --(1,1);
			\end{tikzpicture}
			\caption{$OB$}
			\label{Fig:lac_OB_B0B}
		\end{figure}
		\noindent
The introduction of the line segments $B_0B$ and $OB_1$ allows one to visualize the implementation of the interpolation Lemma \ref{Bourgain}.  In particular, on the line segment $B_0B$ we have the decay estimates for the operator obtained in Lemma \ref{dual:lac} (with the choice $\gamma=\alpha$):
\begin{equation}\label{decay:OB}
				\|M_j^{\alpha}f\|_{L^{\frac{2n}{n-\alpha}}(\R^n)}\lesssim2^{-j\big(\frac{(n-1)(n-\alpha)}{2n}-\frac{\alpha}{n}\big)}\|f\|_{L^{\frac{2n}{n+\alpha}}(\R^n)}.
			\end{equation}	
Meanwhile, we claim that on $OB_1$ one has the growth estimate
			\begin{equation}\label{growth:OB}
				\|M_j^{\alpha}f\|_{L^{r}(\R^n)}\lesssim2^{j\frac{\alpha}{2n}}\|f\|_{L^{\frac{r}{2n-1}}(\R^n)},
			\end{equation}
			for
\begin{align}\label{r_equation}
			r=2q=\frac{2n(n-1)}{\alpha}.
			\end{align}

			If we assume \eqref{growth:OB} for the moment, then the validity of the estimates \eqref{decay:OB} and \eqref{growth:OB} for the operators $T_j=M_j^{\alpha}$ allow one to apply the interpolation Lemma \ref{Bourgain} with $\epsilon_1=\frac{\alpha}{2n}$ and $\epsilon_2=\frac{(n-1)(n-\alpha)}{2n}-\frac{\alpha}{n}$ to obtain a restricted weak-type estimate for the operator $T = \sum_j T_j: L^{p,1}(\mathbb{R}^n) \to L^{q,\infty}(\mathbb{R}^n)$ for
			\begin{align*}
				\frac{1}{p}&=\frac{\theta}{ \frac{r}{2n-1}}+\frac{1-\theta}{\frac{2n}{n+\alpha}}=\frac{\alpha}{n-1},\\
				\frac{1}{q}&=\frac{\theta}{r} +\frac{1-\theta}{\frac{2n}{n-\alpha}}=\frac{\alpha}{n(n-1)},
			\end{align*}
where 
\begin{align}\label{lacunary_theta}
\theta = \frac{(n-1)(n-\alpha)-2\alpha}{(n-1)(n-\alpha)-\alpha}.
\end{align}
			\subsection*{Open line segment $AB$.}We follow a similar argument to that for the line segment $OB$. Fix $0<\alpha<\frac{n(n-1)}{n+1}$. Then $\big(\frac{1}{p},\frac{1}{q}\big)\in AB$ if $p=\frac{n(n-1)}{n(n-1)-\alpha}$ and $q=\frac{n-1}{n-1-\alpha}$. One again has the decay estimate \eqref{decay:OB} on the segment $B_0B$, while we claim that on $AB_1$ one has the growth estimate
			\begin{equation}\label{growth:AB}
				\|M_j^{\alpha}f\|_{L^{\left(\frac{r}{2n-1}\right)'(\R^n)}}\lesssim2^{j\frac{\alpha}{2n}}\|f\|_{L^{r'}(\R^n)}.
			\end{equation}
			Assuming \eqref{growth:AB} holds true, a similar application of the interpolation Lemma \ref{Bourgain}, utilizing the family of estimates \eqref{decay:OB} and \eqref{growth:AB}, with $\epsilon_1=\frac{\alpha}{2n}$ and $\epsilon_2=\frac{(n-1)(n-\alpha)}{2n}-\frac{\alpha}{n}$ gives a restricted weak-type estimate for 
			\begin{align*}
				\frac{1}{p}&= \frac{\theta}{r'}+\frac{1-\theta}{\frac{2n}{n+\alpha}}=\frac{n(n-1)-\alpha}{n(n-1)},\\
				\frac{1}{q}&=\frac{\theta}{(\frac{r}{2n-1})^\prime}+\frac{1-\theta}{\frac{2n}{n-\alpha}}=\frac{n-1-\alpha}{n-1},
			\end{align*}
where again $r$ is as defined in \eqref{r_equation} and $\theta$ is as defined in \eqref{lacunary_theta}.
This completes the proof of Theorem \ref{thm:lac} except for the claims \eqref{growth:OB} and \eqref{growth:AB}.

			\subsubsection*{Proof of Claims \eqref{growth:OB} and \eqref{growth:AB}.}Let $r=\frac{2n(n-1)}{\alpha}$, as defined in \eqref{r_equation}. We recall from Lemma \ref{decay:lac}, we have the estimates
			\begin{align}
				\|M_j^0f\|_{L^r(\R^n)}&\lesssim2^{-j\frac{\alpha}{2n}}\|f\|_{L^r(\R^n)},\label{claim1}\\
				\|M_j^0f\|_{L^{r'}(\R^n)}&\lesssim2^{-j\frac{\alpha}{2n}}\|f\|_{L^{r'}(\R^n)},\label{claim2}
			\end{align}
while the choice $\gamma=\frac{n(n-1)-\alpha}{n-1}$ in Lemma \ref{dual:lac} gives the estimate
			\begin{equation}\label{conjugatediagonal}
				\|M_j^{\frac{n(n-1)-\alpha}{n-1}}f\|_{L^r(\R^n)}\lesssim2^{j\big(1-\frac{n+1}{r}\big)}\|f\|_{L^{r'}(\R^n)}.
			\end{equation}
			Complex interpolation of \eqref{claim1} and \eqref{conjugatediagonal}, \eqref{claim2} and \eqref{conjugatediagonal}, yields the desired inequalities, \eqref{growth:OB} and \eqref{growth:AB}, respectively.
		\end{proof}
		
		\section{Proof of Theorem \ref{thm:full}} \label{full}

				We have the following estimates for $M_{*,j}^{\alpha}$ defined in \eqref{LP_maximal_0} and \eqref{LP_maximal}.
		\begin{lemma}\label{Pointwise:full}
			Let $n\geq2,j\geq0$ and $0\leq\alpha\leq n$, then we have
			\[M_{*,j}^{\alpha}f(x)\lesssim 2^j\mathcal{M}^{\alpha}f(x).\]
		\end{lemma}
		The proof of Lemma \ref{Pointwise:full} is similar to that of Lemma \ref{Pointwise:lac}, so we omit the details for brevity.
		\begin{lemma}\label{full:0}
			Let $n\geq2$ and $j\geq 1$, then we have
			\begin{enumerate}
				\item When $1<p\leq2$, $\|M_{*,j}^0f\|_{L^p(\R^n)}\lesssim2^{-j\big(n-1-\frac{n}{p}\big)}\|f\|_{L^p(\R^n)}$.
				\item When $p\geq2$, $\|M_{*,j}^0f\|_{L^p(\R^n)}\lesssim2^{-j\frac{n-2}{p}}\|f\|_{L^p(\R^n)}$.
			\end{enumerate}
		\end{lemma}
		\begin{proof}
			When $1<p\leq2$, $L^p-$estimates are given in \cite[page 481]{Grafakos}. For $p>2$, we interpolate $L^2-$estimate from $(1)$ with trivial $L^\infty-$bound of $M_{*,j}^0$ to obtain the desired claim.
		\end{proof}
		
		\begin{lemma}\label{full:alpha}
			Let $n\geq3,j\geq1$ and $0\leq\alpha\leq n$. Then, the following holds.
			\begin{enumerate}
				\item $\|M_{*,j}^\alpha f\|_{L^{\frac{2n}{n-\alpha}}(\R^n)}\lesssim2^{-j\big(n-1-\frac{n+\alpha}{2}\big)}\|f\|_{L^{\frac{2n}{n+\alpha}}(\R^n)}$.
				\item  When $\frac{n}{n+1}\leq\alpha\leq n$,
				$$\|M_{*,j}^{\alpha}f\|_{L^{\frac{2n^2}{(n-\alpha)(n-1)}}(\R^n)}\lesssim2^{-j\frac{n(n-1)^2-(n^2+1)\alpha}{2n^2}}\|f\|_{L^{\frac{2n^2}{n^2-n+(n+1)\alpha}}(\R^n)}.$$
			\end{enumerate}
		\end{lemma}
		\begin{proof}
			The endpoint estimate for fractional maximal function $\mathcal{M}^\alpha: L^1(\mathbb{R}^n) \to L^{\frac{n}{n-\alpha},\infty}(\R^n)$ and  Lemma \ref{Pointwise:full} imply
			\begin{equation}\label{full:endpoint}
				\|M_{*,j}^\alpha f\|_{L^{\frac{n}{n-\alpha},\infty}(\R^n)}\lesssim2^{j}\|f\|_{L^{1}(\R^n)}.
			\end{equation}
Complex interpolation of the estimate \eqref{full:endpoint} for $\alpha=n$ with $L^2-$estimate from Lemma \ref{full:0}(1) proves $(1)$.
			
			To prove $(2)$, we recall estimate (1.10) from \cite{Lee1}.
			\[\Big\|M_{0,j}^{0}f\Big\|_{L^{\frac{2(n+1)}{n-1}}(\R^n)}\lesssim2^{-j\frac{n^2-2n-1}{2(n+1)}}\|f\ast\psi_{2^{j}}\|_{L^2(\R^n)}.\]
			For general $k$, by rescaling the above estimate, we have
			\[\Big\|M_{k,j}^{0}f\Big\|_{L^{\frac{2(n+1)}{n-1}}(\R^n)}\lesssim2^{-k\frac{n}{n+1}}2^{-j\frac{n^2-2n-1}{2(n+1)}}\|f\ast\psi_{2^{j-k}}\|_{L^2(\R^n)}.\]
Therefore
			\begin{align}\label{novern+1}
				\Big\|M_{k,j}^{\frac{n}{n+1}}f\Big\|_{L^{\frac{2(n+1)}{n-1}}(\R^n)}&\lesssim2^{k\frac{n}{n+1}}\Big\|M_{k,j}^{0}f\Big\|_{L^{\frac{2(n+1)}{n-1}}(\R^n)}\nonumber\\
				&\lesssim2^{-j\frac{n^2-2n-1}{2(n+1)}}\|f\ast\psi_{2^{j-k}}\|_{L^2(\R^n)}.
			\end{align}
			The inequality \eqref{novern+1} in combination with the embedding $\ell_2\hookrightarrow\ell_{\frac{2(n+1)}{n-1}}\hookrightarrow\ell_\infty$ yields			\begin{align*}
				\|M_{*,j}^{\frac{n}{n+1}}f\|_{L^{\frac{2(n+1)}{n-1}}(\R^n)}&\leq \bigg(\sum_{k\in\Z}\Big\|M_{k,j}^{\alpha}f\Big\|_{L^{\frac{2(n+1)}{n-1}}(\R^n)}^{\frac{2(n+1)}{n-1}}\bigg)^{\frac{n-1}{2(n+1)}}\\
				&\lesssim \bigg(\sum_{k\in\Z}2^{-j\frac{n^2-2n-1}{n-1}}\|f\ast\psi_{2^{j-k}}\|_{L^2(\R^n)}^{\frac{2(n+1)}{n-1}}\bigg)^{\frac{n-1}{2(n+1)}}\\
				&\lesssim 2^{-j\frac{n^2-2n-1}{2(n+1)}}\bigg(\sum_{k\in\Z}\|f\ast\psi_{2^{j-k}}\|_{L^2(\R^n)}^{2}\bigg)^{\frac{1}{2}}\\
				&\lesssim 2^{-j\frac{n^2-2n-1}{2(n+1)}}\|f\|_{L^2(\R^n)}.
			\end{align*}
			Complex interpolation of the above $L^2\to L^{\frac{2(n+1)}{n-1}}-$estimate of $M_{*,j}^{\frac{n}{n+1}}$ with \eqref{full:endpoint} for $\alpha=n$ yields $(2)$.
		\end{proof} 
		We now complete the proof of Theorem \ref{thm:full}.
		\begin{proof}[Proof of Theorem \ref{thm:full}]
			We will prove restricted weak-type estimates on each line segment separately.
			\subsection*{Open line segment $PQ$.}Fix $0<\alpha<n-2$ and note that if $(\frac{1}{p},\frac{1}{q})$ lie on $PQ$, then $p=\frac{n}{n-1}$ and $q=\frac{n}{n-1-\alpha}$.  As in the lacunary case, any such point lies on a line with unit slope through two segments we introduce in the following figure.
			
\begin{figure}[H]
			\centering
			\begin{tikzpicture}[scale=8]
				\tiny
				(0,0)--(0.8,0.8)--(0.8,0.2)--(20/26,4/26)--cycle;
				\draw[thin][->]  (0,0)node[left]{$O$} --(1.15,0) node[right]{$\frac{1}{p}$};
				\draw[thin][->]  (0,0) --(0,1.2) node[left]{$\frac{1}{q}$};
				\draw[thin] (0,1)node[left]{$(0,1)$}--(1,1)node[right]{$(1,1)$};
				\draw[densely dotted] (0.5,0.5)node[left]{$Q_0$}--(1,0);
				\draw [densely dotted] (.9,.9)node[above]{$P_1$} --(.9,0.1)node[below]{$Q_1$};
				
				\draw[thin] (0.675,0.325) --(.9,.55);
				\draw [densely dotted] (1,0)node[below]{$(1,0)$} --(1,1);
				\draw[thin] (20/26,4/26)node[below]{$R$}--(0.8,0.2)node[right]{$Q$}--(0.8,0.8)node[above]{$P$}--(0,0)--cycle;
				\node at (0.8,0.8) {\tiny{$\circ$}};
				\node at (0.8,0.45) {\tiny{$\circ$}};
				\node at (0.68,0.47) {\tiny{$\left(\frac{n-1}{n},\frac{n-1-\alpha}{n}\right)$}};
				\node at (4/5,1/5) {\tiny{$\circ$}};
				\node at (20/26,4/26) {\tiny{$\circ$}};
				\draw[thin] (0,0) --(0.8,0.8);
			\end{tikzpicture}
			\caption{$PQ$.}\label{Fig:PQ}
		\end{figure}	

			Complex interpolation of the estimates \eqref{full:endpoint} and the estimates in Lemma \ref{full:alpha}(1) yield growth estimates on the line segment $P_1Q_1$:
			\[\|M_{*,j}^\alpha f\|_{L^{\frac{2n}{2n-2\alpha-1}}(\R^n)}\lesssim2^{\frac{j}{2}}\|f\|_{L^{\frac{2n}{2n-1}}(\R^n)}.\]
 Lemma \ref{full:alpha}(1) provides the decay estimates on the line segment $Q_0Q$.  These bounds allow one to apply the interpolation Lemma \ref{Bourgain} with $\epsilon_1=\frac{1}{2}$ and $\epsilon_2=n-1-\frac{n+\alpha}{2}$ to obtain restricted weak-type bounds for 
			\begin{align*}
				\frac{1}{p}=& \frac{\theta}{\frac{2n}{2n-1}}+\frac{1-\theta}{\frac{2n}{n+\alpha}}=\frac{n-1}{n},\\
				\frac{1}{q}=& \frac{\theta}{\frac{2n}{2n-1-2\alpha}}+\frac{1-\theta}{\frac{2n}{n-\alpha}}=\frac{n-1-\alpha}{n},
			\end{align*}
where
\begin{align}\label{full_theta}
\theta = \frac{n-2-\alpha}{n-1-\alpha}.
\end{align}
			
			\subsection*{Open line segment $OR$}Observe that if $(\frac{1}{p},\frac{1}{q})$ lie on $OR$, then $p=\frac{n-1}{\alpha}$ and $q=\frac{n(n-1)}{\alpha}$ for $0<\alpha<\frac{n(n-1)^2}{n^2+1}$.   This point again lies on a line with unit slope through two new line segments we introduce in the following figure.
						
			\begin{figure}[H]
			\centering
			\begin{tikzpicture}[scale=8]
				\tiny
				(0,0)--(0.8,0.8)--(0.8,0.2)--(20/26,4/26)--cycle;
				\draw[thin][->]  (0,0)node[left]{$O$} --(1.15,0) node[right]{$\frac{1}{p}$};
				\draw[thin][->]  (0,0) --(0,1.2) node[left]{$\frac{1}{q}$};
				\draw[thin] (0,1)node[left]{$(0,1)$}--(1,1)node[right]{$(1,1)$};
				\draw[densely dotted] (0.5,1/3)node[left]{$R_0$}--(1,0);
				\draw [densely dotted] (0,0) --(44/50,4/50)node[below]{$R_1$};
				
				\draw[thin] (0.385,0.035) --(.61,.26);
				\draw [densely dotted] (1,0)node[below]{$(1,0)$} --(1,1);
				\draw[thin] (20/26,4/26)node[below]{$R$}--(0.8,0.2)node[right]{$Q$}--(0.8,0.8)node[above]{$P$}--(0,0)--cycle;
				\node at (0.8,0.8) {\tiny{$\circ$}};
				\node at (1.75/4,1.75/20) {\tiny{$\circ$}};
				\node at (0.33,0.12) {\tiny{$\left(\frac{\alpha}{n-1},\frac{\alpha}{n(n-1)}\right)$}};
				\node at (4/5,1/5) {\tiny{$\circ$}};
				\node at (20/26,4/26) {\tiny{$\circ$}};
				\draw[thin] (0,0) --(0.8,0.8);
			\end{tikzpicture}
			\caption{$OR$.}
			\label{Fig:OR}
		\end{figure}

			We recall the estimate from Lemma \ref{full:alpha}(2) for $\alpha=n-1$,
			\begin{equation}\label{growth:n-1}
				\|M_{*,j}^{n-1}f\|_{L^{\frac{2n^2}{n-1}}(\R^n)}\lesssim2^{j\frac{n^2-1}{2n^2}}\|f\|_{L^{\frac{2n^2}{2n^2-n-1}}(\R^n)}.
			\end{equation}
			Using complex interpolation between \eqref{growth:n-1} and trivial estimate $\|M_{*,j}^0f\|_{L^{\infty}}\lesssim\|f\|_{L^{\infty}}$, we obtain
			\begin{equation}\label{growth:OR}
				\|M_{*,j}^{\alpha}f\|_{L^{\frac{2n^2}{\alpha}}(\R^n)}\lesssim2^{j\frac{(n+1)\alpha}{2n^2}}\|f\|_{L^{\frac{2n^2}{(2n+1)\alpha}}(\R^n)}.
			\end{equation}
An application of the interpolation Lemma \ref{Bourgain} between \eqref{growth:OR} and Lemma \ref{full:alpha}(2) with $\epsilon_1=\frac{(n+1)\alpha}{2n^2}$ and $\epsilon_2=\frac{n(n-1)^2-(n^2+1)\alpha}{2n^2}$ proves restricted weak-type bounds for
			\begin{align*}
				\frac{1}{p}=& \frac{\theta}{\frac{2n^2}{(2n+1)\alpha}}+\frac{1-\theta}{\frac{2n^2}{n^2-n+(n+1)\alpha}}=\frac{\alpha}{n-1},\\
				\frac{1}{q}=&\frac{\theta}{\frac{2n^2}{\alpha}}+\frac{1-\theta}{\frac{2n^2}{(n-1)(n-\alpha)}}=\frac{\alpha}{n(n-1)},
			\end{align*}
where
\begin{align}
\theta = \frac{n(n-1)^2-(n^2+1)\alpha}{n(n-1)(n-1-\alpha)}
\end{align}
			when $\frac{n}{n+1}\leq\alpha<\frac{n(n-1)^2}{n^2+1}$. 
			
			For $0<\alpha<\frac{n}{n+1}$, we use complex interpolation between the trivial $L^\infty(\R^n)\to L^\infty(\R^n)$-estimate of $A^{0}_{*}$ and restricted weak-type bound of $A^{\frac{n}{n+1}}_{*}$ for $p=\frac{n^2-1}{n}, q=n^2-1$, which is proved in the above paragraph.
			
			\subsection*{Open line segment $QR$} Note that if if $(\frac{1}{p},\frac{1}{q})\in QR$, then $n-2<\alpha<\frac{n(n-1)^2}{n^2+1},p=\frac{2n}{(n-1)(n-\alpha)}$ and $q=\frac{2n}{n(n-1)-(n+1)\alpha}$.  This point again lies on a line with unit slope through two new line segments we introduce in the following figure, which is an enlargement of Figure \ref{Fig:OR} around the segment $QR$.

		\begin{figure}[H]
			\centering
			\begin{tikzpicture}[scale=8]
				\tiny

				\draw[thin] (0,7/50)--(0.3,0.2)node[below]{$R$}--(0.7,4/5)node[right]{$Q$}--(0.7,.9);
				\draw[densely dotted] (0,2/5)node[left]{$R_0$}--(0.3,0.2);
				\draw[densely dotted] (0.7,4/5)--(1,0.5)node[left]{$Q_1$};
				\draw[thin] (1/4,7/30)--(91/120,89/120);
				\node at (28/60,27/60) {\tiny{$\circ$}};
				\node at (28/60-.2,27/60+.1) {$\left(\frac{(n-1)(n-\alpha)}{2n},\frac{n(n-1)-(n+1)\alpha}{2n}\right)$};

			\end{tikzpicture}
			\caption{$QR$.}
			\label{Fig:QR}
		\end{figure}
On the line segment $R_0R$, Lemma \ref{full:alpha}(2) gives decay estimates, while on the line segment 	$QQ_1$ Lemma \ref{full:alpha}(1) gives growth estimates.  Thus we may again apply Lemma \ref{Bourgain} with $\epsilon_1=\frac{\alpha-n+2}{2}$ and $\epsilon_2=\frac{n(n-1)^2-(n^2+1)\alpha}{2n^2}$ to obtain the restricted weak-type estimate for
			\begin{align*}
				\frac{1}{p}=&\frac{\theta}{\frac{2n}{n+\alpha}}+\frac{1-\theta}{\frac{2n^2}{n^2-n+(n+1)\alpha}}=\frac{(n-1)(n-\alpha)}{2n},\\
				\frac{1}{q}=&\frac{\theta}{\frac{2n}{n-\alpha}}+\frac{1-\theta}{\frac{2n^2}{(n-1)(n-\alpha)}}=\frac{n(n-1)-(n+1)\alpha}{2n},
			\end{align*}
where
\begin{align}\label{last_theta}
\theta = \frac{n(n-1)^2-(n^2+1)\alpha}{n-\alpha}.
\end{align}
	
A summary of the applications of the interpolation Lemma \ref{Bourgain} is depicted in the following table.
			\begin{table}[ht]
				\centering
				\begin{tabular}{| M{.6cm}| M{2.3cm}| M{1cm}| M{1.2cm}| M{1.9cm}| M{1.6cm}| M{2.3cm}| }
					\hline
					Line & Range of $\alpha$  & $p_1$ & $q_1$ & $p_2$ & $q_2$ & $\theta$ \\ 
					\hline
				 &&&&& &	\\
					
					$PQ$ & $(0,n-2)$ & $\frac{2n}{2n-1}$ &$\frac{2n}{2n-1-2\alpha}$ & $\frac{2n}{n+\alpha}$ & $\frac{2n}{n-\alpha}$ & $\frac{n-2-\alpha}{n-1-\alpha}$  \\[10pt]
					\hline 
					 &&&&& &	\\
					$OR$ & $(\frac{n}{n+1},\frac{n(n-1)^2}{n^2+1})$ & $\frac{2n^2}{(2n+1)\alpha}$ & $\frac{2n^2}{\alpha}$ & $\frac{2n^2}{n^2-n+(n+1)\alpha}$ & $\frac{2n^2}{(n-1)(n-\alpha)}$ & $\frac{n(n-1)^2-(n^2+1)\alpha}{n(n-1)(n-1-\alpha)}$ \\
					&&&&& &	\\
					\hline
					 &&&&& &	\\
					$QR$ & $\big(n-2,\frac{n(n-1)^2}{n^2+1}\big)$ & $\frac{2n}{n+\alpha}$ & $\frac{2n}{n-\alpha}$ & $\frac{2n^2}{n^2-n+(n+1)\alpha}$ & $\frac{2n^2}{(n-1)(n-\alpha)}$ & $\frac{n(n-1)^2-(n^2+1)\alpha}{n-\alpha}$ \\
					&&&&& &	\\
					\hline
					
				\end{tabular}
			\end{table}

		\end{proof}
		
		\section{Proof of Theorem \ref{thm:full2}}\label{full2}
		As in the proof of Theorem \ref{thm:full}, we have the following decomposition of $A^\alpha_{*}$.
		\begin{align*}
			A_{*}^{\alpha}f(x)\leq\sum_{j\geq0}M_{*,j}^{\alpha}f(x),
		\end{align*}
		and recall that $M_{*,j}^{\alpha}f(x)\lesssim2^j\mathcal{M}^{\alpha}f(x)$.
		\begin{lemma}\label{decay:full2}
			Let $\frac{2}{7}<\alpha\leq2$, then
			\[\|M_{*,j}^{\alpha}f\|_{L^{\frac{8}{2-\alpha}}(\R^2)}\lesssim2^{j\frac{5\alpha-2}{8}}\|f\|_{L^{\frac{8}{2+3\alpha}}(\R^2)}.\]
		\end{lemma}
		\begin{proof}
			By rescaling the estimate (1.5) from \cite{Lee1}, we have
			\[\Big\|M_{k,j}^{0}f\Big\|_{L^{q}(\R^2)}\lesssim2^{-k(\frac{2}{p}-\frac{2}{q})}2^{j(\frac{3}{2p}-\frac{1}{2q}-\frac{1}{2})}\|f\ast\psi_{2^{j-k}}\|_{L^p(\R^2)}.\]
			for $q>\frac{14}{3}$ and $\frac{1}{p}+\frac{3}{q}=1$.
			
			If $p$ and $q$ satisfy the equations $\frac{1}{p}-\frac{1}{q}=\frac{\alpha}{2}$ and $\frac{1}{p}+\frac{3}{q}=1$, then $p=\frac{8}{2+3\alpha}$ and $q=\frac{8}{2-\alpha}$. When $q>\frac{14}{3}$, we have $\alpha>\frac{2}{7}$. Hence, for $\alpha>\frac{2}{7}$, we obtain
			\begin{align}\label{dim2}
				\Big\|M_{k,j}^{\alpha}f\Big\|_{L^{\frac{8}{2-\alpha}}(\R^2)}\lesssim&2^{k\alpha}\Big\|M_{k,j}^{0}f\Big\|_{L^{\frac{8}{2-\alpha}}(\R^2)}\nonumber\\
				\lesssim&2^{j\frac{5\alpha-2}{8}}\|f\ast\psi_{2^{j-k}}\|_{L^{\frac{8}{2+3\alpha}}(\R^2)}.
			\end{align}    
			For $\frac{2}{7}<\alpha\leq\frac{2}{3}$, we have $p=\frac{8}{2+3\alpha}\geq2$. Then, using \eqref{dim2} and the embedding $\ell_p\hookrightarrow\ell_q\hookrightarrow\ell_\infty$ for $p\leq q<\infty$, we have
			\begin{align}\label{dim22}
				\|M_{*,j}^{\alpha}f\|_{L^{\frac{8}{2-\alpha}}(\R^2)}&\leq \bigg(\sum_{k\in\Z}\Big\|M_{k,j}^{\alpha}f\Big\|_{L^{\frac{8}{2-\alpha}}(\R^2)}^{\frac{8}{2-\alpha}}\bigg)^{\frac{2-\alpha}{8}}\nonumber\\
				&\lesssim \bigg(\sum_{k\in\Z}2^{j\frac{5\alpha-2}{2-\alpha}}\Big\|f\ast\psi_{2^{j-k}}\Big\|_{L^{\frac{8}{2+3\alpha}}(\R^2)}^{\frac{8}{2-\alpha}}\bigg)^{\frac{2-\alpha}{8}}\nonumber\\
				&\lesssim 2^{j\frac{5\alpha-2}{8}}\|f\|_{L^{\frac{8}{2+3\alpha}}(\R^2)}.
			\end{align}
			where we used the Littlewood-Paley theory in the last inequality. This completes the proof for $\frac{2}{7}<\alpha\leq\frac{2}{3}$.
			
			For the remaining range $\frac{2}{3}<\alpha\leq2$, the desired estimate follows using a complex interpolation between \eqref{full:endpoint} for $\alpha=2$ and \eqref{dim22} for $\alpha=\frac{2}{3}$.
		\end{proof}
		
		\begin{lemma}\label{full2:interpolate}
			Let $0\leq\alpha\leq1$ and $j\geq1$, then the following holds.
			\begin{enumerate}
				\item $\|M_{*,j}^{\alpha}f\|_{L^{\frac{8}{\alpha}}(\R^2)}\lesssim2^{j\frac{3\alpha}{8}}\|f\|_{L^{\frac{8}{5\alpha}}(\R^2)}$.
				\item $\|M_{*,j}^{\alpha}f\|_{L^{\frac{8}{4-3\alpha}}(\R^2)}\lesssim2^{j\frac{3\alpha}{8}}\|f\|_{L^{\frac{8}{4+\alpha}}(\R^2)}$.
				\item When $0\leq\alpha\leq\frac{1}{3}$, $\|M_{*,j}^{\alpha}f\|_{L^{\frac{8}{4-7\alpha}}(\R^2)}\lesssim2^{-j\frac{\alpha}{8}}\|f\|_{L^{\frac{8}{4-3\alpha}}(\R^2)}$.
			\end{enumerate}
		\end{lemma}
		\begin{proof}
			The estimates in $(1)$ and $(2)$ can be obtained by applying complex interpolation between Lemma \ref{decay:full2} with $\alpha=1$ and Lemma \ref{full:0} (where $\alpha=0$) for $p=\infty$ and $p=2$ respectively.
			
			To obtain $(3)$, we use complex interpolation between the estimates in Lemma \ref{full:0} (where $\alpha=0$) for $p=2$ and Lemma \ref{decay:full2} with $\alpha=\frac{1}{3}$.
		\end{proof}
		Finally, we complete the proof of Theorem \ref{thm:full2} using the $L^p-$estimates of $M_{*,j}^{\alpha}$.
		\begin{proof}[Proof of Theorem \ref{thm:full2}.]
			First, note that if $\big(\frac{1}{p},\frac{1}{q}\big)\in \Delta OPR$, then $\alpha$ varies from $0$ to $\frac{2}{5}$. We will prove restricted weak type estimates on open line segments $PR$ and $OR$ separately. 
			\subsection*{Open line segment $PR$.} If $\Big(\frac{1}{p},\frac{1}{q}\Big)\in PR$, then $p=\frac{4}{2-\alpha}$ and $q=\frac{4}{2-3\alpha}$.  The use of Lemma \ref{Bourgain} to obtain estimates for points on $PR$ is divided into two cases, as depicted in the following figure.
			
\begin{figure}[H]
		\centering
			\begin{tikzpicture}[scale=6]
				\tiny
				(0,0)--(0.5,0.5)--(2/5,1/5)--cycle;
				\draw[thin][->]  (0,0)node[left]{$O$} --(1.15,0) node[right]{$\frac{1}{p}$};
				\draw[thin][->]  (0,0) --(0,1.2) node[left]{$\frac{1}{q}$};
				\draw [densely dotted] (11/10,7/10)node[right]{$P_1$} --(1,1);
				\draw[thin] (4/5,2/5)node[right]{$R$}--(1,1)node[above]{$P$}--(0,0)--cycle;
				\draw[densely dotted] (10/14,6/14)node[left]{$R_0$}--(3/4,5/12)node[below]{$R_0'$} --(1,1/3);
				\draw[thin] (3/4+1/8,5/12+7/24)--(3/4+1/8+.168,5/12+7/24+.168);
				\draw[thin] (31/40,49/120)--(31/40+12.65/40,49/120+12.65/40);
				\draw[densely dotted] (3/4,5/12) --(1,1);
				\node[right] at (1,1) {$(\frac{1}{2},\frac{1}{2})$};
				\node at (10/14,6/14) {\tiny{$\circ$}};
				\node at (1,1) {\tiny{$\circ$}};
				\node at (4/5,2/5) {\tiny{$\circ$}};
				\node at (3/4,5/12) {\tiny{$\circ$}};
				\draw[thin] (0,0) --(0.5,0.5);
			\end{tikzpicture}
			\caption{$PR$}	
		\end{figure}	
			
			When $0<\alpha\leq\frac{1}{3}$, the estimates we utilize are on the segment $R_0'P$ and $PP_1$, that is, we apply the interpolation Lemma \ref{Bourgain} using the estimates in Lemma \ref{full2:interpolate}(2) and Lemma \ref{full2:interpolate}(3) with $\epsilon_1=\frac{3\alpha}{8}$ and $\epsilon_2=\frac{\alpha}{8}$ to obtain restricted weak-type bounds for
			\begin{align*}
				\frac{1}{p}=& \frac{\theta}{\frac{8}{4+\alpha}}+\frac{1-\theta}{\frac{8}{4-3\alpha}}=\frac{2-\alpha}{4},\\
				\frac{1}{q}=&\frac{\theta}{\frac{8}{4-3\alpha}}+\frac{1-\theta}{\frac{8}{4-7\alpha}}=\frac{2-3\alpha}{4},
			\end{align*}
where $\theta = \frac{1}{4}$.

For $\frac{1}{3}<\alpha<\frac{2}{5}$, we use directly the estimates on $R_0'R$ along with the same estimate on $PP_1$ and apply the same interpolation.  Precisely, we apply the Lemma \ref{Bourgain} with the bounds in Lemma \ref{full2:interpolate}(2) and Lemma \ref{decay:full2} with $\epsilon_1=\frac{3\alpha}{8}$ and $\epsilon_2=\frac{2-5\alpha}{8}$ to obtain restricted weak-type bounds for
\begin{align*}
				\frac{1}{p}=& \frac{\theta}{\frac{8}{4+\alpha}}+\frac{1-\theta}{\frac{8}{2+3\alpha}}=\frac{2-\alpha}{4},\\
				\frac{1}{q}=&\frac{\theta}{\frac{8}{4-3\alpha}}+\frac{1-\theta}{\frac{8}{2-\alpha}}=\frac{2-3\alpha}{4},
			\end{align*}
where 
\begin{align}\label{lacunary_theta}
\theta = \frac{2-5\alpha}{2(1-\alpha)}.
\end{align}

			\subsection*{Line segment $OR$.} If $\Big(\frac{1}{p},\frac{1}{q}\Big)\in OR$, then $p=\frac{1}{\alpha}$ and $q=\frac{2}{\alpha}$. Here the interpolation picture is similar to Figure \ref{Fig:OR}, where we use estimates on the line segment $OR$ and $R_0R$ to obtain estimates for an interval on the segment $OR$ near $R$.  Precisely, for $\frac{2}{7}<\alpha<\frac{2}{5}$, we apply the interpolation Lemma \ref{Bourgain} utilizing the estimates in Lemma \ref{full2:interpolate}(1) and Lemma \ref{decay:full2} with $\epsilon_1=\frac{3\alpha}{8}$ and $\epsilon_2=\frac{2-5\alpha}{8}$ to obtain restricted weak-type bounds for
			\begin{align*}
				\frac{1}{p}=&\frac{\theta}{\frac{8}{5\alpha}}+\frac{1-\theta}{\frac{8}{2+3\alpha}}=\alpha,\\
				\frac{1}{q}=&\frac{\theta}{\frac{8}{\alpha}}+\frac{1-\theta}{\frac{8}{2-\alpha}}=\frac{\alpha}{2},
			\end{align*}
where $\theta$ is as defined in \eqref{lacunary_theta}.

			When $0<\alpha\leq\frac{2}{7}$, we apply complex interpolation between the trivial $L^\infty(\R^n)\to L^\infty(\R^n)$-estimate of $A^{0}_{lac}$ and the restricted weak-type bound of $A^{\frac{1}{3}}_{*}$ for $p=3, q=6$, which is proved in the above paragraph.

A summary of the applications of Lemma \ref{Bourgain} for $n=2$ is depicted in the following table.		
		\begin{table}[ht]
			\centering
			\begin{tabular}{| M{.6cm}| M{2.2cm}| M{1cm}| M{1.2cm}| M{1.5cm}| M{1.4cm}| M{2cm}| }
				\hline
				Line & Range of $\alpha$  & $p_1$ & $q_1$ & $p_2$ & $q_2$ & $\theta$ \\
				\hline
					&&&&& &	\\
				$PR$ & $0\leq \alpha<\frac{1}{3}$ & $\frac{8}{4+\alpha}$ & $\frac{8}{4-3\alpha}$ & $\frac{8}{4-3\alpha}$ & $\frac{8}{4-7\alpha}$  & $\frac{1}{4}$ \\
					&&&&& &	\\
				\hline
					&&&&& &	\\
				& $\frac{1}{3}\leq \alpha<\frac{2}{5}$ & $\frac{8}{4+\alpha}$ & $\frac{8}{4-3\alpha}$ & $\frac{8}{2+3\alpha}$ & $\frac{8}{2-\alpha}$  & $\frac{2-5\alpha}{2-2\alpha}$ \\
				$OR$	&&&&& &	\\
				\cline{2-7}
					&&&&& &	\\
				 & $\frac{2}{7}\leq \alpha<\frac{2}{5}$ & $\frac{8}{5\alpha}$ & $\frac{8}{\alpha}$ & $\frac{8}{2+3\alpha}$ & $\frac{8}{2-\alpha}$  & $\frac{2-5\alpha}{2-2\alpha}$ \\
					&&&&& &	\\
				\hline
			\end{tabular}
		\end{table}
		
		\end{proof}
				
		\section*{Acknowledgements}
		The authors would like to thank David Beltran for comments on a draft of the manuscript.  
		R.~Basak is supported by the National Science and Technology Council of Taiwan under research grant number 113-2811-M-003-014.  
		S. Choudhary is supported by the National Center for Theoretical Sciences through National Science and Technology Council of Taiwan under research grant number 115-2124-M-002-009-.
		D. Spector is supported by the National Science and Technology Council of Taiwan under research grant number 113-2115-M-003-017-MY3 and the Taiwan Ministry of Education under the Yushan Fellow Program.
		


\begin{bibdiv}
			\begin{biblist}
				
				\bib{BasakSpector}{article}{
					author={Basak, Riju},
					author={Spector, Daniel},
					title={Estimates for the wave equation on $\beta$-dimensional spaces of measures},
					journal={},
					volume={},
					date={},
					number={},
					pages={},
					issn={},
					review={},
					doi={https://doi.org/10.48550/arXiv.2512.12618},
				}
				
				\bib{BRS}{article}{
					author={Beltran, David},
					author={Ramos, Jo\~ao Pedro},
					author={Saari, Olli},
					title={Regularity of fractional maximal functions through Fourier
						multipliers},
					journal={J. Funct. Anal.},
					volume={276},
					date={2019},
					number={6},
					pages={1875--1892},
					issn={0022-1236},
					review={\MR{3912794}},
					doi={10.1016/j.jfa.2018.11.004},
				}
				
				
				\bib{Bourgain}{article}{
					author={Bourgain, Jean},
					title={Estimations de certaines fonctions maximales},
					language={French, with English summary},
					journal={C. R. Acad. Sci. Paris S\'er. I Math.},
					volume={301},
					date={1985},
					number={10},
					pages={499--502},
					issn={0249-6291},
					review={\MR{0812567}},
				}
				
				\bib{Bourgain1986}{article}{
   author={Bourgain, J.},
   title={Averages in the plane over convex curves and maximal operators},
   journal={J. Analyse Math.},
   volume={47},
   date={1986},
   pages={69--85},
   issn={0021-7670},
   review={\MR{0874045}},
   doi={10.1007/BF02792533},
}
				
				\bib{Calderon_lacunary}{article}{
					author={Calder\'on, Calixto P.},
					title={Lacunary spherical means},
					journal={Illinois J. Math.},
					volume={23},
					date={1979},
					number={3},
					pages={476--484},
					issn={0019-2082},
					review={\MR{0537803}},
				}
				
				\bib{CladekKrause}{article}{
					author={Cladek, Laura},
					author={Krause, Benjamin},
					title={Improved endpoint bounds for the lacunary spherical maximal
						operator},
					journal={Anal. PDE},
					volume={17},
					date={2024},
					number={6},
					pages={2011--2032},
					issn={2157-5045},
					review={\MR{4776291}},
					doi={10.2140/apde.2024.17.2011},
				}
				
				\bib{Christ}{article}{
					author={Christ, Michael},
					title={Weak type $(1,1)$ bounds for rough operators},
					journal={Ann. of Math. (2)},
					volume={128},
					date={1988},
					number={1},
					pages={19--42},
					issn={0003-486X},
					review={\MR{0951506}},
					doi={10.2307/1971461},
				}
				
				\bib{CoifmanWeiss}{article}{
					author={Coifman, R. R.},
					author={Weiss, Guido},
					title={Book Review: Littlewood-Paley and multiplier theory},
					journal={Bull. Amer. Math. Soc.},
					volume={84},
					date={1978},
					number={2},
					pages={242--250},
					issn={0002-9904},
					review={\MR{1567040}},
					doi={10.1090/S0002-9904-1978-14464-4},
				}
				
				
				
				\bib{CGG}{article}{
					author={Cowling, Michael},
					author={Garc\'ia-Cuerva, Jos\'e},
					author={Gunawan, Hendra},
					title={Weighted estimates for fractional maximal functions related to
						spherical means},
					journal={Bull. Austral. Math. Soc.},
					volume={66},
					date={2002},
					number={1},
					pages={75--90},
					issn={0004-9727},
					review={\MR{1922609}},
					doi={10.1017/S0004972700020694},
				}
				
				\bib{DRF}{article}{
					author={Duoandikoetxea, Javier},
					author={Rubio de Francia, Jos\'e L.},
					title={Maximal and singular integral operators via Fourier transform
						estimates},
					journal={Invent. Math.},
					volume={84},
					date={1986},
					number={3},
					pages={541--561},
					issn={0020-9910},
					review={\MR{0837527}},
					doi={10.1007/BF01388746},
				}
				
				\bib{Grafakos}{book}{
					author={Grafakos, Loukas},
					title={Classical Fourier analysis},
					series={Graduate Texts in Mathematics},
					volume={249},
					edition={3},
					publisher={Springer, New York},
					date={2014},
					pages={xviii+638},
					isbn={978-1-4939-1193-6},
					isbn={978-1-4939-1194-3},
					review={\MR{3243734}},
					doi={10.1007/978-1-4939-1194-3},
				}
				
				\bib{KS}{article}{
					author={Kinnunen, Juha},
					author={Saksman, Eero},
					title={Regularity of the fractional maximal function},
					journal={Bull. London Math. Soc.},
					volume={35},
					date={2003},
					number={4},
					pages={529--535},
					issn={0024-6093},
					review={\MR{1979008}},
					doi={10.1112/S0024609303002017},
				}
				
				
				
				
				\bib{Lacey}{article}{
					author={Lacey, Michael T.},
					title={Sparse bounds for spherical maximal functions},
					journal={J. Anal. Math.},
					volume={139},
					date={2019},
					number={2},
					pages={613--635},
					issn={0021-7670},
					review={\MR{4041115}},
					doi={10.1007/s11854-019-0070-2},
				}
				
				
				
				\bib{Lee1}{article}{ 
					AUTHOR = {Lee, Sanghyuk},
					TITLE = {Endpoint estimates for the circular maximal function},
					JOURNAL = {Proc. Amer. Math. Soc.},
					FJOURNAL = {Proceedings of the American Mathematical Society},
					VOLUME = {131},
					YEAR = {2003},
					NUMBER = {5},
					PAGES = {1433--1442},
					ISSN = {0002-9939,1088-6826},
					MRCLASS = {42B25 (35L05)},
					MRNUMBER = {1949873},
					MRREVIEWER = {J.\ M.\ Aldaz},
					DOI = {10.1090/S0002-9939-02-06781-3},
					URL = {https://doi.org/10.1090/S0002-9939-02-06781-3},
				}
				
				
				\bib{Oberlin}{article}{
					author={Oberlin, Daniel M.},
					title={Operators interpolating between Riesz potentials and maximal
						operators},
					journal={Illinois J. Math.},
					volume={33},
					date={1989},
					number={1},
					pages={143--152},
					issn={0019-2082},
					review={\MR{0974016}},
				}
				
					\bib{RoosSeeger}{article}{
					author={Roos, Joris},
					author={Seeger, Andreas},
					title={Problems on Spherical Maximal Functions},
					journal={},
					volume={},
					date={},
					pages={},
					issn={},
					review={},
					doi={https://arxiv.org/pdf/2511.11283},
				}
				
				
				
				\bib{Sagher}{article}{
					author={Sagher, Yoram},
					title={On analytic families of operators},
					journal={Israel J. Math.},
					volume={7},
					date={1969},
					pages={350--356},
					issn={0021-2172},
					review={\MR{0257822}},
					doi={10.1007/BF02788866},
				}
				
				
				
				
				\bib{SchlagThesis}{book}{
   author={Schlag, Wilhelm},
   title={L('p) to L('q) estimates for the circular maximal function},
   note={Thesis (Ph.D.)--California Institute of Technology},
   publisher={ProQuest LLC, Ann Arbor, MI},
   date={1996},
   pages={79},
   review={\MR{2694231}},
}

				
				\bib{Schlag}{article}{
					author={Schlag, W.},
					title={A generalization of Bourgain's circular maximal theorem},
					journal={J. Amer. Math. Soc.},
					volume={10},
					date={1997},
					number={1},
					pages={103--122},
					issn={0894-0347},
					review={\MR{1388870}},
					doi={10.1090/S0894-0347-97-00217-8},
				}
				
				\bib{Schlag-Sogge}{article}{
					author={Schlag, Wilhelm},
					author={Sogge, Christopher D.},
					title={Local smoothing estimates related to the circular maximal theorem},
					journal={Math. Res. Lett.},
					volume={4},
					date={1997},
					number={1},
					pages={1--15},
					issn={1073-2780},
					review={\MR{1432805}},
					doi={10.4310/MRL.1997.v4.n1.a1},
				}
				
				
				\bib{STW}{article}{
					author={Seeger, Andreas},
					author={Tao, Terence},
					author={Wright, James},
					title={Pointwise convergence of lacunary spherical means},
					conference={
						title={Harmonic analysis at Mount Holyoke},
						address={South Hadley, MA},
						date={2001},
					},
					book={
						series={Contemp. Math.},
						volume={320},
						publisher={Amer. Math. Soc., Providence, RI},
					},
					isbn={0-8218-2903-3},
					date={2003},
					pages={341--351},
					review={\MR{1979950}},
					doi={10.1090/conm/320/05617},
				}
				
				\bib{Stein}{article}{
					author={Stein, Elias M.},
					title={Maximal functions. I. Spherical means},
					journal={Proc. Nat. Acad. Sci. U.S.A.},
					volume={73},
					date={1976},
					number={7},
					pages={2174--2175},
					issn={0027-8424},
					review={\MR{0420116}},
					doi={10.1073/pnas.73.7.2174},
				}
				
				
				\bib{SteinWeiss}{book}{
					author={Stein, Elias M.},
					author={Weiss, Guido},
					title={Introduction to Fourier analysis on Euclidean spaces},
					series={Princeton Mathematical Series},
					volume={No. 32},
					publisher={Princeton University Press, Princeton, NJ},
					date={1971},
					pages={x+297},
					review={\MR{0304972}},
				}
				
			\end{biblist}
		\end{bibdiv}
	\end{document}